\documentclass[a4paper,11pt]{article}

\usepackage{amsmath}
\usepackage{amssymb}
\usepackage{amsthm}
\usepackage{graphicx}
\usepackage{amscd}
\usepackage{epic, eepic}
\usepackage{url}
\usepackage{color}
\usepackage[utf8]{inputenc} 
\usepackage{comment}
\usepackage{enumerate}   
\usepackage{todonotes}
\usepackage{graphicx}
\usepackage{epstopdf}
\usepackage{enumitem}
\usepackage{tikz}
\usepackage[top=24mm, bottom=25mm, left=25mm, right=25mm]{geometry}
\usepackage[bookmarks=false,hidelinks]{hyperref}

\makeatletter
\renewenvironment{proof}[1][\proofname] {\par\pushQED{\qed}\normalfont\topsep6\p@\@plus6\p@\relax\trivlist\item[\hskip\labelsep\bfseries#1\@addpunct{.}]\ignorespaces}{\popQED\endtrivlist\@endpefalse}
\makeatother

\newtheorem{theorem}{\bf Theorem}[section]
\newtheorem{lemma}[theorem]{\bf Lemma}

\theoremstyle{definition}

\newtheorem{definition}[theorem]{\bf Definition}

\def\eps{\varepsilon}

\title{Resolution of the Erd\H os--Sauer problem on regular subgraphs}

\author{Oliver Janzer\thanks{Department of Mathematics, ETH Z\"urich, Switzerland. Email: \href{mailto:oliver.janzer@math.ethz.ch} {\nolinkurl{oliver.janzer@math.ethz.ch}}.
	Research supported by an ETH Z\"urich Postdoctoral Fellowship 20-1 FEL-35.}
	\and Benny Sudakov\thanks{Department of Mathematics, ETH Z\"urich, Switzerland. Email:
	\href{mailto:benny.sudakov@gmail.com} {\nolinkurl{benny.sudakov@gmail.com}}.
	Research supported in part by SNSF grant 200021\_196965.}
}

\date{}

\begin{document}

\maketitle

\begin{abstract}
In this paper we completely resolve the well-known problem of Erd\H os and Sauer from 1975 which asks for the maximum number of edges an $n$-vertex graph can have without containing a $k$-regular subgraph, for some fixed integer $k\geq 3$. We prove that any $n$-vertex graph with average degree at least $C_k\log \log n$ contains a $k$-regular subgraph. This matches the lower bound of Pyber, R\"odl and Szemer\'edi and substantially improves an old result of Pyber, who showed that average degree at least $C_k\log n$ is enough.

Our method can also be used to settle asymptotically a problem raised by Erd\H os and Simonovits in 1970 on almost regular subgraphs of sparse graphs and to make progress on the well-known question of Thomassen from 1983 on finding subgraphs with large girth and large average degree.
\end{abstract}

\section{Introduction}

The problem of finding regular subgraphs in graphs has a very long history. Note that finding a $1$-regular subgraph is the same as finding a matching. One of the oldest results of graph theory is Petersen's theorem from 1891 \cite{Pet1891}, which states that every cubic, bridgeless graph contains a perfect matching. The celebrated Hall's theorem \cite{Hall35} gives a necessary and sufficient condition for a bipartite graph to have a perfect matching, while Tutte's theorem \cite{Tut47} gives such a condition for an arbitrary graph. Later, Tutte \cite{Tut52} found a necessary and sufficient condition for a graph to contain a $k$-regular spanning subgraph, and more generally, for it to contain an $f$-factor -- that is, a spanning subgraph in which each vertex $v$ has degree $f(v)$.

The problem of finding general, not necessarily spanning, regular subgraphs was also extensively studied. In 1975, Erd\H os and Sauer \cite{Erd75} asked the following natural extremal question. Given a positive integer $k$, what is the maximum number of edges that an $n$-vertex graph can have if it does not contain a $k$-regular subgraph?  The problem also appeared in Bollob\'as's book on extremal graph theory \cite{Bol78} and in the book of Bondy and Murty \cite{BM76}. Later, Erd\H os \cite{Erd81} mentioned this as one of his favourite unsolved problems. Let us write $f_k(n)$ for the smallest number of edges which guarantees a $k$-regular subgraph. Trivially $f_2(n)=n$, but already for $k=3$, the answer is unclear. Erd\H os and Sauer observed that $f_3(n)=O(n^{8/5})$ follows from the known upper bound for the Tur\'an number of the cube \cite{ES70} and suggested that $f_k(n)\leq n^{1+\eps}$ for any fixed $\eps>0$ and sufficiently large $n$.

A very influential algebraic technique to find regular subgraphs was developed 
in the early 80's by Alon, Friedland and Kalai. In \cite{AFK84short},
motivated by a conjecture of Berge and  Sauer, they showed that any $4$-regular multigraph plus an edge contains a $3$-regular subgraph. Alon, Friedland and Kalai \cite{AFK84} also proved 
several results which state, roughly speaking, that nearly regular graphs have regular subgraphs with not too small degree.
One of these results, together with an ingenious argument, was used by Pyber \cite{Pyb85} in 1985 to show that indeed $f_k(n)\leq n^{1+\eps}$. More precisely, Pyber proved that any $n$-vertex graph with average degree at least $C_k\log n$ contains a $k$-regular subgraph (here and in the rest of the paper logarithms are to the base two).

At this point, the best lower bound was due to Chv\'atal \cite{Erd75}, who had shown that $f_3(2n+3)> 6n$. This was greatly improved by Pyber, R\"odl and Szemer\'edi \cite{PRSz95}, who found a remarkable construction of  graphs with a superlinear number of edges that do not contain $k$-regular subgraphs.

\begin{theorem}[Pyber--R\"odl--Szemer\'edi \cite{PRSz95}] \label{thm:PRSzlower}
    There is some $c>0$ such that for every $n$ there exists an $n$-vertex graph with at least $cn\log \log n$ edges which does not contain a $k$-regular subgraph for any $k\geq 3$.
\end{theorem}

\noindent
Furthermore, Pyber, R\"odl and Szemer\'edi generalized Pyber's result to show that for any positive integer $k$, there is a constant $C=C(k)$ such that any graph with maximum degree~$\Delta$ and average degree at least $C\log \Delta$ contains a $k$-regular subgraph.

Despite substantial interest from many researchers, the above bounds have not been improved in the last 30 years. At the same time, several variants of the original problem have been considered. Bollob\'as, Kim and Verstra\"ete \cite{BKV06} studied the threshold for a random graph to have a $k$-regular subgraph. R\"odl and Wysocka \cite{RW97} investigated the largest $r$ for which every $n$-vertex graph with at least $\gamma n^2$ edges has an $r$-regular subgraph. Many papers have been written on the existence of regular subgraphs in hypergraphs; see, e.g., \cite{MV09,DHL+,KK14,Kim16,HK18}.

In this paper we prove the following result, which completely resolves the problem of Erd\H os and Sauer.

\begin{theorem} \label{thm:logloglog error term}
    For any positive integer $k$, there is a constant $C=C(k)$ such that any graph with maximum degree $\Delta$ and average degree at least $C\log \log \Delta$ contains a $k$-regular subgraph. In particular, any $n$-vertex graph with average degree at least $C \log \log n$ contains such a subgraph.
\end{theorem}

Our results also make progress on two other old and well-known problems. The first one, due to Erd\H os and Simonovits from 1970 \cite{ES70}, concerns the question of how dense an almost-regular subgraph there must exist in a graph with $n$ vertices and $n\log n$ edges. An almost-regular graph here is one in which the maximum degree and the minimum degree differ by at most a constant factor. Our results resolve this question asymptotically, showing that one can find an almost-regular subgraph with $m=\omega(1)$ vertices and at least $m(\log m)^{1/2-o(1)}$ edges, which is tight by a result of Alon \cite{Alon08}. Since the full discussion of this topic takes considerable space, we postpone it to the concluding remarks (Section \ref{sec:concluding remarks}).

Finally, let us discuss the other application of our main result.
In 1983, Thomassen \cite{Tho83} conjectured that for every $t,g\in \mathbb{N}$, there exists some $d$ such that any graph with average degree at least $d$ contains a subgraph with average degree at least $t$ and girth at least $g$. K\"uhn and Osthus \cite{KO04} proved this for $g\leq 6$ (see also \cite{MPS21} for an improved bound on $d$), but the general case is wide open. It is easy to see that for every $t,g\in \mathbb{N}$, if $d$ is sufficiently large, then any $d$-regular graph has a subgraph with average degree at least $t$ and girth at least $g$. Hence, our Theorem \ref{thm:logloglog error term} implies that any graph $G$ with average degree at least $C(t,g)\log \log \Delta(G)$ contains a subgraph with average degree at least $t$ and girth at least $g$. This improves the best bound towards Thomassen's conjecture, due to Dellamonica and R\"odl \cite{DR11}, which states that any graph $G$ with average degree at least $\alpha(t,g) (\log \log \Delta(G))^{\beta(t,g)}$ contains a subgraph with average degree at least $t$ and girth at least $g$.

\medskip

The rest of the paper is organized as follows. In the next section, we give a sketch of the proof of Theorem \ref{thm:logloglog error term}. In Section \ref{sec:preliminaries}, we prove some simple preliminary lemmas. The key lemma is proved in Section \ref{sec:key lemma}. The proof of Theorem \ref{thm:logloglog error term} is then completed in Section \ref{sec:complete proof}. We give some concluding remarks about large almost-regular subgraphs in Section \ref{sec:concluding remarks}.

\textbf{Notation.} For a graph $G$ and a vertex $u\in V(G)$, $N_G(u)$ denotes the neighbourhood of $u$ in $G$ and we write $d_G(u)$ for the degree of $u$. Given vertices $u$ and $u'$, $d_G(u,u')$ stands for the number of common neighbours of $u$ and $u'$ in $G$. We write $\Delta(G)$ for the maximum degree of $G$.

\section{An overview of the proof}
Although our proof is short, we think that it may be useful to give a sketch of the main ideas. We start by briefly discussing the lower bound construction of Pyber, R\"odl and Szemer\'edi since it partly motivates our argument. Their (random) graph can be described roughly as follows. The vertex set is $A\cup B$, where $|B|=n$, $A$ is the disjoint union of sets $A(j)$ for $\frac{1}{4}\log \log n\leq j\leq \frac{1}{2}\log \log n$ with $|A(j)|=n/2^{2^j}$, and for each $j$ and $v\in B$, $v$ has a unique neighbour in $A(j)$, chosen uniformly at random. Note that a typical vertex in $A(j)$ has degree about $2^{2^j}$.

In a general bipartite graph, we obtain a similar partitioning of part $A$ according to the degrees and show, very roughly speaking, that if there is some $j$ such that each $v\in B$ has at least a large constant number of neighbours in $A(j)$ (unlike in the above construction where each $v\in B$ has only one neighbour there), then $G$ already has a $k$-regular subgraph.

More precisely, let $r$ be a large constant (that can depend on $k$) and let $G$ be a bipartite graph with parts $A$ and $B$ such that for some positive integers $s,t\geq r$, every $v\in B$ has degree $r$, the average degree of a vertex in $A$ is at least $2^s$, the maximum degree of a vertex in $A$ is at most $2^t$ and $t\leq (1+\frac{1}{r-1})s$.
We remark that it is possible (and not too hard) to find such a subgraph in any graph with maximum degree $\Delta$ and average degree at least $100r^2\log \log \Delta$.
For the sake of simplicity, let us also assume that $G$ is $C_4$-free, although this assumption can be significantly relaxed -- we will briefly discuss this later. We shall now argue that $G$ has a $k$-regular subgraph.

Our key lemma (stated and proved in Section \ref{sec:key lemma}) provides a subgraph which, although is not necessarily regular, has much better regularity properties than $G$. More precisely, it asserts the existence of subsets $A''\subset A$, $B''\subset B$ and positive integers $s',t'\geq r$ such that in the graph $G[A''\cup B'']$ all the previous conditions are satisfied (with $A'',B'',s',t'$ in place of $A,B,s,t$), and, additionally, $t'-s'\leq \log(40(t-s)r^2)$. Hence, as long as $t-s> 3\log r$, $t'-s'$ is smaller than $t-s$.
Iteratively applying this lemma, eventually we obtain subsets $A^*\subset A$, $B^*\subset B$ and positive integers $s^*,t^*\geq r$ such that in the graph $G[A^*\cup B^*]$, every $v\in B^*$ has degree $r$, the average degree of a vertex in $A^*$ is at least $2^{s^*}$, the maximum degree of a vertex in $A^*$ is at most $2^{t^*}$ and $t^*-s^*\leq 3\log r$.
The last property means that the maximum degree is at most $2^{3\log r}$ times the average degree in $A^*$. Taking a random subset $\hat{B}\subset B^*$ of size $|A^*|$ and deleting the vertices in $G[A^*\cup \hat{B}]$ whose degree is much larger than expected, we obtain a subgraph with average degree about $r$ and maximum degree at most about $r2^{3\log r}=r^4$. We can then use the result of Pyber, R\"odl and Szemer\'edi (see Theorem~\ref{thm:PRSz}) to find a $k$-regular subgraph in this graph.

We shall now sketch the proof of our key lemma. Let $A_{s+1}=\{u\in A:d_G(u)\leq 2^{s+1}\}$ and for each $s+2\leq i\leq t$, let $A_i=\{u\in A: 2^{i-1}<d_G(u)\leq 2^{i}\}$. Observe that the sets $A_{s+1},\dots,A_t$ partition $A$. For each $u\in A$, let $\alpha(u)$ be the unique $i$ with $u\in A_i$. For $v\in B$, let
$\beta(v)=\sum_{u\in N_G(v)} \alpha(u)$. Note that for every $v\in B$, we have $(s+1)r\leq \beta(v)\leq tr$, so there are at most $(t-s)r$ possible values for $\beta(v)$. Hence, by the pigeon hole principle, there are some $\gamma\geq (s+1)r$ and a subset $\tilde{B}\subset B$ of size at least $\frac{|B|}{(t-s)r}$ such that $\beta(v)=\gamma$ for all $v\in \tilde{B}$.

Let $A'$ be a random subset of $A$ where each $u\in A$ is kept independently with probability $2^{\alpha(u)-t}$. Let $B'=\{v\in \tilde{B}: N_G(v)\subset A'\}$. Let us see why we expect $G[A'\cup B']$ to be nearly biregular. Firstly, every $v\in B'$ has degree precisely $r$ in $G[A'\cup B']$.
Now let $u\in A$. We claim that, conditional on $u\in A'$, $|N_G(u)\cap B'|$ is distributed as a binomial random variable $\textrm{Bin}\big(|N_G(u)\cap \tilde{B}|, 2^{\gamma-\alpha(u)-(r-1)t}\big)$. Indeed, conditional on $u\in A'$, each $v\in N_G(u)\cap \tilde{B}$ belongs to $B'$ with probability $\prod_{w\in N_G(v)\setminus \{u\}} 2^{\alpha(w)-t}=2^{\beta(v)-\alpha(u)-(r-1)t}=2^{\gamma-\alpha(u)-(r-1)t}$ and these events are independent for all $v$ since $N_G(v)\setminus \{u\}$ are pairwise disjoint as $G$ is $C_4$-free. Thus, if $u\in A'$, then the degree of $u$ in $G[A'\cup B']$ is concentrated around $|N_G(u)\cap \tilde{B}|2^{\gamma-\alpha(u)-(r-1)t}$. Since $|N_G(u)\cap \tilde{B}|\leq d_G(u)\leq 2^{\alpha(u)}$, we see that it is very unlikely that the degree of $u$ in $G[A'\cup B']$ is much larger than $2^{\gamma-(r-1)t}$. On the other hand, $\tilde{B}$ has size similar to $B$, so on average $|N_G(u)\cap \tilde{B}|$ is not much smaller than $d_G(u)$, which is in turn at least $2^{\alpha(u)-1}$ (unless $\alpha(u)=s+1$). Hence, it is not hard to see that the average degree in $G[A'\cup B']$ of a vertex in $A'$ is expected to be not much smaller than $2^{\gamma-(r-1)t}$. A small loss arises from the fact that $\tilde{B}$ is potentially smaller than $B$ by a factor of $(t-s)r$; this partly explains why $t'-s'$ can be as large as $\log(40(t-s)r^2)$. We remark that the condition $t\leq (1+\frac{1}{r-1})s$ ensures that $\gamma-(r-1)t\geq sr-(r-1)t+r\geq r$, so the expected average degree of the vertices in $A'$ is not too small.

Now we take $A''$ to be the set of those vertices in $A'$ whose degree in $G[A'\cup B']$ is not much larger than expected (i.e. not much larger than $2^{\gamma-(r-1)t}$) and set $B''=\{v\in B': N_G(v)\subset A''\}$. By the strong concentration of the degrees, $G[A''\cup B'']$ has almost as many edges as $G[A'\cup B']$, so the average degree of a vertex in $A''$ is not much smaller than it was in $A'$. Hence, the desired conditions about $A''$ and $B''$ are satisfied (for suitable $s'\approx \gamma-(r-1)t$ and $t'\approx \gamma-(r-1)t$).

We conclude this outline by remarking that a slightly different argument works even if $G$ is not necessarily $C_4$-free, but the codegrees in $G$ are not extremely large.

\section{Preliminaries} \label{sec:preliminaries}
As we have mentioned in the outline of the proof, one of the conditions in our key lemma is that no two vertices in the graph have very large codegree. We prove that with the loss of a constant factor in the number of edges, we can pass to a subgraph satisfying this property. For this, we will use that we can assume that our graphs are $K_{k,k}$-free (or else, they contain a $k$-regular subgraph).

\begin{lemma} \label{lem:large degrees force K_kk}
    Let $k$ be a positive integer and let $H$ be a $K_{k,k}$-free bipartite graph with parts $S$ and $T$. Assume that $d_H(u)\geq k|T|^{1-1/k}$ holds for every $u\in S$. Then $e(H)\leq k|T|$.
\end{lemma}

\begin{proof}
    Suppose that $e(H)>k|T|$. Then, by Jensen's inequality (where for $x<i$ we define ${x \choose i}=0$),
    $$\sum_{v\in T} \binom{d_H(v)}{k}\geq |T|\binom{e(H)/|T|}{k}\geq |T|\left(\frac{e(H)}{k|T|}\right)^k.$$
    By the degree condition we have $e(H)\geq k|S||T|^{1-1/k}$. Then $\frac{e(H)}{k|T|}\geq |S||T|^{-1/k}$ and therefore,
    $\sum_{v\in T} \binom{d_H(v)}{k} \geq |S|^k\geq k\binom{|S|}{k}$. This implies that there is some $R\subset S$ of size $k$ whose common neighbourhood in $T$ has size at least $k$. Then $H$ contains $K_{k,k}$, which is a contradiction.
\end{proof}

The proof of the next lemma is somewhat similar to that of Lemma 2.3 from \cite{MPS21}.

\begin{lemma} \label{lem:clean large codegrees}
    Let $k$ be a positive integer. Let $G$ be a $K_{k,k}$-free bipartite graph with parts $A$ and $B$ and assume that $d_G(u)\leq m$ holds for all $u\in A$. Then $G$ has a spanning subgraph $G'$ such that $e(G')\geq e(G)/(k+1)$ and $d_{G'}(u,u')\leq k m^{1-1/k}$ for any two distinct $u,u'\in A$.
\end{lemma}

\begin{proof}
    Let $A=\{u_1,u_2,\dots,u_n\}$ and let $G_0=G$. Let us define spanning subgraphs $G_1,\dots,G_n$ of $G$ recursively as follows. Having defined $G_{i-1}$ for some $1\leq i\leq n$, let $T=N_{G_{i-1}}(u_i)$, let $S=\{u_j: j>i \text{ and } d_{G_{i-1}}(u_i,u_j)>k m^{1-1/k}\}$ and let $H=G_{i-1}[S\cup T]$. Since $H$ is a subgraph of $G$, it is $K_{k,k}$-free. Moreover, for any $u_j\in S$, $d_H(u_j)=d_{G_{i-1}}(u_i,u_j)>k m^{1-1/k}\geq k |T|^{1-1/k}$. Hence, by Lemma \ref{lem:large degrees force K_kk}, $e(H)\leq k|T|$. Let $G_i$ be the subgraph of $G_{i-1}$ obtained by deleting the edges in $H$. Then $e(G_{i-1})-e(G_i)=e(H)\leq kd_{G_{i-1}}(u_i)$.
    
    Observe that for any $i<j$, we have $d_{G_n}(u_i,u_j)\leq km^{1-1/k}$.
    Moreover, note that for any $1\leq i\leq n$, $d_{G_n}(u_i)=d_{G_{i-1}}(u_i)$, so
     $$e(G_n)=\sum_{i=1}^n d_{G_n}(u_i)=\sum_{i=1}^n d_{G_{i-1}}(u_i)\geq \sum_{i=1}^n (e(G_{i-1})-e(G_i))/k=(e(G_0)-e(G_n))/k.$$
    Thus, $e(G_n)\geq e(G_0)/(k+1)=e(G)/(k+1)$, which means that $G'=G_n$ satisfies the conditions described in the lemma.
\end{proof}

We now define two notions of ``near-regularity" that will be used in our proofs.

\begin{definition}
    A graph $G$ is called \emph{$K$-almost-regular} if the maximum degree of $G$ is at most $K$ times the minimum degree of $G$.
\end{definition}

\begin{definition} \label{def:almost-biregular}
    We say that a bipartite graph $G$ is \emph{$(L,d)$-almost-biregular} if the following hold. $G$ has parts $A$ and $B$, where $d_G(v)=d$ for every $v\in B$, and, writing $D=e(G)/|A|$, we have $D\geq d$ (equivalently $|A|\leq |B|$) and $d_G(u)\leq LD$ for every $u\in A$.
\end{definition}

Our key lemma will provide an $(L,d)$-almost-biregular subgraph where $L$ is fairly small. The next lemma allows us to find a large almost-regular subgraph in it.

\begin{lemma} \label{lem:almost bireg to almost reg}
    Let $G$ be an $(L,\delta)$-almost-biregular graph for some $L\geq \delta\geq 2$. Then $G$ has a $64$-almost-regular subgraph with average degree at least $\frac{\delta}{16\log L}$.
\end{lemma}

The proof of this lemma uses the following result, which is a slight variant of Lemma~2.7 from \cite{PRSz95}. It states that as long as $L$ is subexponential in the average degree, we can pass to a subgraph with large average degree and constant $L$.

\begin{lemma} \label{lem:from L to 4 almost bireg}
    Let $G$ be an $(L,\delta)$-almost-biregular graph. Suppose that $L\delta\leq 2^{\lfloor \delta/(d-1) \rfloor}$ and $d\leq \delta$. Then $G$ has a $(4,d)$-almost-biregular subgraph.
\end{lemma}

\noindent
Since the proof of this result is almost identical to that of Lemma 2.7 from \cite{PRSz95}, it is omitted here. 

The next simple lemma shows that in an almost-biregular graph we can find an almost-regular subgraph with similar average degree.

\begin{lemma} \label{lem:find nearly reg subgraph}
    Let $L,d\geq 1$ and let $G$ be an $(L,d)$-almost-biregular graph. Then $G$ has a non-empty subgraph $G'$ with average degree at least $d/2$ and maximum degree at most $4Ld$.
\end{lemma}

\begin{proof}
    Choose $A$, $B$ and $D$ according to Definition \ref{def:almost-biregular}. Define $B'$ to be a random subset of $B$ where each $v\in B$ is kept independently with probability $\frac{d}{D}$. Let $G'$ be the subgraph of $G[A\cup B']$ obtained by deleting all edges $uv$ with $|N_G(u)\cap B'|\geq 4Ld$. Clearly, $\Delta(G')\leq \max(4Ld,d)=4Ld$.
    
    Let $X=e(G[A\cup B'])$ and let $Y$ be the number of edges $uv\in E(G[A\cup B'])$ with $|N_G(u)\cap B'|\geq 4Ld$. Now
    $$\mathbb{E}[X]=\sum_{uv\in E(G)} \mathbb{P}(v\in B')=e(G)\frac{d}{D}$$
    and
    \begin{align*}
        \mathbb{E}[Y]
        &=\sum_{uv\in E(G)} \mathbb{P}(v\in B' \text{ and } |N_G(u)\cap B'|\geq 4Ld) \\
        &=\sum_{uv\in E(G)} \mathbb{P}(v\in B')\mathbb{P}\left(|N_G(u)\cap B'|\geq 4Ld \hspace{1mm} | \hspace{1mm} v\in B'\right).
    \end{align*}
    For any $uv\in E(G)$, we have $$\mathbb{E}[|N_G(u)\cap B'| \hspace{1mm} | \hspace{1mm} v\in B']=1+(d_G(u)-1)\frac{d}{D}\leq 1+LD\frac{d}{D}= 1+Ld\leq 2Ld,$$ so it follows by Markov's inequality that
    $$\mathbb{P}\left(|N_G(u)\cap B'|\geq 4Ld \hspace{1mm} | \hspace{1mm} v\in B'\right)\leq 1/2.$$
    Hence,
    $$\mathbb{E}[Y]\leq \sum_{uv\in E(G)}\mathbb{P}(v\in B')/2=e(G)\frac{d}{2D}.$$
    Thus, using $|A|D=|B|d=e(G)$,
    \begin{align*}
        \mathbb{E}[X-Y-(|A|+|B'|)d/4]
        &\geq e(G)\frac{d}{D}-e(G)\frac{d}{2D}-|A|d/4-|B|\frac{d}{D}d/4 \\
        &= e(G)\frac{d}{D}-e(G)\frac{d}{2D}-e(G)\frac{d}{4D}-e(G)\frac{d}{4D}=0.
    \end{align*}
    In particular, there is an outcome for which $X-Y\geq (|A|+|B'|)d/4$. Since $e(G')=X-Y$, this means that $G'$ has average degree at least $d/2$.
\end{proof}

Lemma \ref{lem:almost bireg to almost reg} can now be proven using the last two lemmas.

\begin{proof}[Proof of Lemma \ref{lem:almost bireg to almost reg}]
    Let $d=\lceil \frac{\delta}{4\log L}\rceil$. Now $\delta/(d-1)\geq 4\log L$, so $\lfloor \delta/(d-1)\rfloor \geq 2\log L$. Hence, $2^{\lfloor \delta/(d-1)\rfloor}\geq L^2\geq L\delta$. By Lemma \ref{lem:from L to 4 almost bireg}, $G$ has a $(4,d)$-almost-biregular subgraph $G'$. By Lemma \ref{lem:find nearly reg subgraph}, $G'$ has a subgraph with average degree at least $d/2$ and maximum degree at most $16d$. Repeatedly discarding vertices of degree less than $d/4$, we end up with a non-empty subgraph $G''$ with minimum degree at least $d/4$ and maximum degree at most $16d$. Clearly, $G''$ is $64$-almost-regular. Moreover, the average degree of $G''$ is at least $d/4\geq \frac{\delta}{16\log L}$.
\end{proof}

In order to find a $k$-regular subgraph in an almost-regular graph with sufficiently large average degree, one can use the result of Pyber, R\"odl and Szemer\'edi.

\begin{theorem}[Pyber--R\"odl--Szemer\'edi \cite{PRSz95}] \label{thm:PRSz}
    For any positive integer $k$, there is a constant $C=C(k)$ such that any graph with maximum degree $\Delta$ and average degree at least $C \log \Delta$ contains a $k$-regular subgraph.
\end{theorem}

\section{The key lemma} \label{sec:key lemma}

In this section, we prove the following lemma, which is the main ingredient in our proof. Given a bipartite graph with parts $A$ and $B$ in which $B$ is regular, the maximum degree in $A$ is at most about $1+\eps$ power of the average degree in $A$ and the codegrees are not extremely large, the lemma provides an induced subgraph with much better regularity properties. This condition bounding the maximum degree in terms of the average degree of $A$ is hidden in the upper bound on the codegrees: for the codegree condition to hold, we must have $t\leq (1+\frac{1}{r-1})s$.

\begin{lemma} \label{lem:random regularization}
	Let $r,s,t$ be positive integers such that $s<t$. Let $G$ be a bipartite graph with parts $A$ and $B$ such that $d_G(v)=r$ for every $v\in B$, $d_G(u)\leq 2^t$ for every $u\in A$, $d_G(u,u')\leq 2^{rs-(r-1)t}$ for any two distinct $u,u'\in A$ and $e(G)\geq 2^s|A|$.
    
    Then there are subsets $A''\subset A$ and $B''\subset B$ such that $N_G(v)\subset A''$ for every $v\in B''$ and, writing $G'=G[A''\cup B'']$ and $d'=e(G')/|A''|$, we have $d'\geq \frac{2^{rs-(r-1)t}}{10(t-s)r}$ and $d_{G'}(u)\leq 40(t-s)r^2 d'$ for all $u\in A''$.
\end{lemma}

\begin{proof}
Let $A_{s+1}=\{u\in A:d_G(u)\leq 2^{s+1}\}$ and for each $s+2\leq i\leq t$, let $A_i=\{u\in A: 2^{i-1}<d_G(u)\leq 2^{i}\}$. Observe that the sets $A_{s+1},\dots,A_t$ partition $A$. For each $u\in A$, let $\alpha(u)$ be the unique $i$ with $u\in A_i$. For $v\in B$, let
$$\beta(v)=\sum_{u\in N_G(v)} \alpha(u).$$
Note that for every $v\in B$, we have $(s+1)r\leq \beta(v)\leq tr$, so there are at most $(t-s)r$ possible values for $\beta(v)$. Hence, by the pigeon hole principle, there are some $\gamma\geq (s+1)r$ and a subset $\tilde{B}\subset B$ of size at least $\frac{|B|}{(t-s)r}$ such that $\beta(v)=\gamma$ for all $v\in \tilde{B}$.

Let $A'$ be a random subset of $A$ where each $u\in A$ is kept independently with probability $2^{\alpha(u)-t}$. Let $B'=\{v\in \tilde{B}: N_G(v)\subset A'\}$. 
Let $A''$ be the subset of $A'$ consisting of those vertices $u$ with $|N_G(u)\cap B'|\leq 4r2^{\gamma-(r-1)t}$ and let $B''=\{v\in B': N_G(v)\subset A''\}$. Let $G'=G[A''\cup B'']$. 
Then $d_{G'}(u)\leq 4r2^{\gamma-(r-1)t}$ for all $u\in A''$.

The following claim shows that we expect $E(G[A'\cup B'])\setminus E(G')$ to be small.

\medskip

\noindent \emph{Claim.} For any $uv\in E(G[A\cup \tilde{B}])$,
$$\mathbb{P}(u\in A',v\in B' \text{ and } |N_G(u)\cap B'|\geq 4r2^{\gamma-(r-1)t})\leq \mathbb{P}(v\in B')/(2r).$$

\medskip

\noindent \emph{Proof of Claim.} Note that
\begin{align}
    &\mathbb{E}[|N_G(u)\cap B'| \hspace{1mm} | \hspace{1mm} v\in B']
    =\sum_{w\in N_G(u)\cap \tilde{B}} \mathbb{P}(w\in B' \hspace{1mm} | \hspace{1mm} v\in B') \nonumber \\
    &= \sum_{\substack{w\in N_G(u)\cap \tilde{B}: \\ N_G(w)\cap N_G(v)=\{u\}}} \mathbb{P}(w\in B' \hspace{1mm} | \hspace{1mm} v\in B') + \sum_{\substack{w\in N_G(u)\cap \tilde{B}: \\ N_G(w)\cap N_G(v)\neq \{u\}}} \mathbb{P}(w\in B' \hspace{1mm} | \hspace{1mm} v\in B'). \label{eqn:conditional expectation}
\end{align}
We now bound the two sums separately.
If some $w\in N_G(u)\cap \tilde{B}$ satisfies $N_G(w)\cap N_G(v)=\{u\}$, then \begin{align*}
    \mathbb{P}(w\in B' \hspace{1mm} | \hspace{1mm} v\in B')
    &=\mathbb{P}(w\in B' \hspace{1mm} | \hspace{1mm} u\in A')=\prod_{z\in N_G(w)\setminus \{u\}} \mathbb{P}(z\in A')=\prod_{z\in N_G(w)\setminus \{u\}} 2^{\alpha(z)-t} \nonumber \\
    &=2^{\sum_{z\in N_G(w)\setminus \{u\}} \alpha(z)-(r-1)t}=2^{\beta(w)-\alpha(u)-(r-1)t}=2^{\gamma-\alpha(u)-(r-1)t}.
\end{align*}
Since $d_G(u) \leq 2^{\alpha(u)}$, we have that
$$\sum_{\substack{w\in N_G(u)\cap \tilde{B}: \\ N_G(w)\cap N_G(v)=\{u\}}} \mathbb{P}(w\in B' \hspace{1mm} | \hspace{1mm} v\in B')\leq d_G(u)2^{\gamma-\alpha(u)-(r-1)t}\leq 2^{\gamma-(r-1)t}.$$
On the other hand, the number of $w\in N_G(u)\cap \tilde{B}$ with $N_G(w)\cap N_G(v)\neq \{u\}$ is at most $r2^{rs-(r-1)t}$. Indeed, $v$ has $r$ neighbours in $G$, so there are at most $r$ ways to choose a vertex $u'\in N_G(w)\cap N_G(v)\setminus \{u\}$, and, by assumption, any such $u'$ has at most $2^{rs-(r-1)t}$ common neighbours with $u$. Now using $\gamma\geq (s+1)r$, we have $2^{\gamma-(r-1)t}\geq 2^{r}2^{rs-(r-1)t}\geq r2^{rs-(r-1)t}$. This implies that
$$\sum_{\substack{w\in N_G(u)\cap \tilde{B}: \\ N_G(w)\cap N_G(v)\neq \{u\}}} \mathbb{P}(w\in B' \hspace{1mm} | \hspace{1mm} v\in B')\leq r2^{rs-(r-1)t}\leq 2^{\gamma-(r-1)t}.$$

Thus, by (\ref{eqn:conditional expectation}),
$$\mathbb{E}[|N_G(u)\cap B'| \hspace{1mm} | \hspace{1mm} v\in B']\leq 2^{\gamma-(r-1)t+1}.$$
This implies, by Markov's inequality, that
$$\mathbb{P}[|N_G(u)\cap B'|\geq 4r2^{\gamma-(r-1)t} \hspace{1mm} | \hspace{1mm} v\in B']\leq 1/(2r).$$
Thus,
\begin{align*}
	&\mathbb{P}(u\in A',v\in B' \text{ and } |N_G(u)\cap B'|\geq 4r2^{\gamma-(r-1)t}) \\
	&=\mathbb{P}(v\in B' \text{ and } |N_G(u)\cap B'|\geq 4r2^{\gamma-(r-1)t}) \\
	&=\mathbb{P}(v\in B')\mathbb{P}(|N_G(u)\cap B'|\geq 4r2^{\gamma-(r-1)t} \hspace{1mm} | \hspace{1mm} v\in B')\leq \mathbb{P}(v\in B')/(2r),
\end{align*}
completing the proof of the claim. $\Box$

\medskip

For any $uv\in E(G[A\cup \tilde{B}])$,
\begin{equation}
    \mathbb{P}(u\in A', v\in B')=\mathbb{P}(v\in B')=\prod_{w\in N_G(v)} \mathbb{P}(w\in A')=\prod_{w\in N_G(v)} 2^{\alpha(w)-t}=2^{\gamma-rt}. \label{eqn:prob of edge}
\end{equation}

Let $X=e(G[A'\cup B'])$ and let $Y$ be the number of edges $uv\in E(G[A'\cup B'])$ with $|N_G(u)\cap B'|\geq 4r2^{\gamma-(r-1)t}$. Note that $e(G')\geq X-rY$. Indeed, $E(G')$ is obtained from $E(G[A'\cup B'])$ by deleting all edges incident to vertices $v\in B'$ which have a neighbour $u$ with $|N_G(u)\cap B'|> 4r2^{\gamma-(r-1)t}$. Clearly, there are at most $Y$ such vertices $v\in B'$, so at most $rY$ edges are deleted. By equation (\ref{eqn:prob of edge}) and the Claim, we have
\begin{align*}
    &\mathbb{E}[X-rY] \\
    &=\sum_{uv\in E(G[A\cup \tilde{B}])} \left(\mathbb{P}(u\in A', v\in B')-r\mathbb{P}(u\in A', v\in B' \text{ and } |N_G(u)\cap B'|\geq 4r2^{\gamma-(r-1)t})\right) \\
    &\geq \sum_{uv\in E(G[A\cup \tilde{B}])} (\mathbb{P}(v\in B')-\mathbb{P}(v\in B')/2) = e(G[A\cup \tilde{B}])2^{\gamma-rt-1} \\
    &=|\tilde{B}|r2^{\gamma-rt-1}\geq \frac{|B|}{(t-s)r}r 2^{\gamma-rt-1}= \frac{e(G)}{(t-s)r}2^{\gamma-rt-1}.
\end{align*}

Note that for every $u\in A$, we have $2^{\alpha(u)}\leq 2^{s+1}+2d_G(u)$. Indeed, this is trivial if $\alpha(u)=s+1$, and else $d_G(u)\geq 2^{\alpha(u)-1}$. Also recall that $e(G)\geq 2^s|A|$. Therefore,
\begin{align*}
    \mathbb{E}\lbrack |A'|\rbrack
    &=\sum_{u\in A}\mathbb{P}(u\in A')=\sum_{u\in A} 2^{\alpha(u)-t}\leq 2^{-t}\sum_{u\in A} (2^{s+1}+2d_G(u)) \\
    &=2^{-t}(|A|2^{s+1}+2e(G))\leq 2^{-t}4e(G).
\end{align*}
Hence,
$$\mathbb{E}\left[X-rY-|A'|\frac{2^{\gamma-(r-1)t}}{10(t-s)r}\right]>0.$$

It follows that there exists an outcome for which $X-rY-|A'|\frac{2^{\gamma-(r-1)t}}{10(t-s)r}>0$. Then $e(G')\geq X-rY\geq |A'|\frac{2^{\gamma-(r-1)t}}{10(t-s)r}$. Hence, $d'=e(G')/|A''|$ satisfies $d'\geq e(G')/|A'|\geq \frac{2^{\gamma-(r-1)t}}{10(t-s)r}\geq \frac{2^{rs-(r-1)t}}{10(t-s)r}$ and $d_{G'}(u)\leq 4r2^{\gamma-(r-1)t}\leq 40(t-s)r^2d'$ for all $u\in A''$, completing the proof.
\end{proof}

The next lemma is obtained by iterative applications of Lemma \ref{lem:random regularization}.

\begin{lemma} \label{lem:iterated key lemma}
    Let $r,s,t$ be positive integers such that $r$ is sufficiently large and $s<t$. Let $G$ be a bipartite graph with parts $A$ and $B$ such that $d_G(v)=r$ for every $v\in B$, $d_G(u)\leq 2^t$ for every $u\in A$, $d_G(u,u')\leq 2^{2rs-(2r-1)t}$ for any two distinct $u,u'\in A$ and $e(G)\geq 2^s|A|$.
    
    Then there exist positive integers $s^*\geq 2rs-(2r-1)t$, $t^*>s^*$ and sets $A^*\subset A$, $B^*\subset B$ such that $N_G(v)\subset A^*$ for every $v\in B^*$, $t^*-s^*\leq 5\log r$ and, writing $G^*=G[A^*\cup B^*]$ and $d^*=e(G^*)/|A^*|$, we have $d^*\geq 2^{s^*}$ and $d_{G^*}(u)\leq 2^{t^*}$ for all $u\in A^*$.
\end{lemma}

\begin{proof}
    We prove the lemma by induction on $t-s$ (with $r$ fixed). If $t-s\leq 5\log r$, then we can take $s^*=s$, $t^*=t$, $A^*=A$ and $B^*=B$. Assume now that $t-s>5\log r$. Note that $2rs-(2r-1)t\leq rs-(r-1)t$, so for any two distinct $u,u'\in A$, we have $d_G(u,u')\leq 2^{rs-(r-1)t}$. By Lemma~\ref{lem:random regularization}, there are subsets $A'\subset A$ and $B'\subset B$ such that $N_G(v)\subset A'$ for every $v\in B'$ and, writing $G'=G[A'\cup B']$ and $d'=e(G')/|A'|$, we have $d'\geq \frac{2^{rs-(r-1)t}}{10(t-s)r}$ and $d_{G'}(u)\leq 40(t-s)r^2 d'$ for all $u\in A'$. By the codegree condition, $2rs-(2r-1)t\geq 0$, so $rs-(r-1)t\geq (t-s)r$ and hence $d'\geq 2$. Let $s'=\lfloor \log d' \rfloor$ and let $t'=\lceil \log(40(t-s)r^2d') \rceil$. Then $s'<t'$, $d_{G'}(v)=r$ for every $v\in B'$, $d_{G'}(u)\leq 2^{t'}$ for every $u\in A'$ and $e(G')\geq 2^{s'}|A'|$.
    Moreover, using $t'\geq rs-(r-1)t$,
    \begin{align*}
        2rs'-(2r-1)t'
        &=t'-2r(t'-s')\geq t'-2r(\log(40(t-s)r^2)+2) \\
        &\geq rs-(r-1)t-2r\log(160(t-s)r^2) \\
        &=2rs-(2r-1)t+r(t-s)-2r\log(160(t-s)r^2) \\
        &\geq 2rs-(2r-1)t,
    \end{align*}
    where the last inequality uses that $t-s\geq 5\log r$ and that $r$ is sufficiently large.
    Hence, for any distinct $u,u'\in A'$, we have $d_{G'}(u,u')\leq d_G(u,u')\leq 2^{2rs-(2r-1)t}\leq 2^{2rs'-(2r-1)t'}$. Finally, $t'-s'\leq \log(40(t-s)r^2)+2<t-s$.
    
    Thus, by the induction hypothesis, there exist positive integers $s^*\geq 2rs'-(2r-1)t'$, $t^*>s^*$ and sets $A^*\subset A'$, $B^*\subset B'$ such that $N_{G'}(v)\subset A^*$ for every $v\in B^*$, $t^*-s^*\leq 5\log r$ and, writing $G^*=G'[A^*\cup B^*]=G[A^*\cup B^*]$ and $d^*=e(G^*)/|A^*|$, we have $d^*\geq 2^{s^*}$ and $d_{G^*}(u)\leq 2^{t^*}$ for all $u\in A^*$. Since $N_G(v)=N_{G'}(v)$ for every $v\in B'$, we have $N_{G}(v)\subset A^*$ for every $v\in B^*$. As $2rs'-(2r-1)t'\geq 2rs-(2r-1)t$, it follows that $s^*\geq 2rs-(2r-1)t$, so $s^*$, $t^*$, $A^*$ and $B^*$ are suitable.
\end{proof}

\section{Completing the proof} \label{sec:complete proof}

The next lemma is the upshot of what we have proved so far.

\begin{lemma} \label{lem:small codeg implies k-regular}
    Let $r,s,t$ be positive integers such that $r$ is sufficiently large and $s<t$. Let $G$ be a bipartite graph with parts $A$ and $B$ such that $d_G(v)=r$ for every $v\in B$, $d_G(u)\leq 2^t$ for every $u\in A$, $d_G(u,u')\leq 2^{2rs-(2r-1)t-r}$ for any two distinct $u,u'\in A$ and $e(G)\geq 2^s|A|$.
    
    Then $G$ contains a $64$-almost-regular subgraph with average degree at least $\frac{r}{80\log r}$.
\end{lemma}

\begin{proof}
    By the assumption on the codegrees, we have $2rs-(2r-1)t-r\geq 0$, so $2rs-(2r-1)t\geq r$. Hence, by Lemma~\ref{lem:iterated key lemma}, $G$ has a $(2^{5\log r},r)$-almost-biregular subgraph $G^*$ (the condition $2rs-(2r-1)t\geq r$ ensures that $s^*\geq r$ and so $d^*\geq r$). Then, by Lemma \ref{lem:almost bireg to almost reg} with $L=2^{5\log r}$ and $\delta=r$, $G^*$ has a $64$-almost-regular subgraph with average degree at least $\frac{r}{80\log r}$.
\end{proof}

The next result is very similar to Lemma \ref{lem:small codeg implies k-regular} -- the main difference is that the codegree condition is no longer present. This can be achieved using Lemma~\ref{lem:clean large codegrees}.

\begin{lemma} \label{lem:key lemma cor}
    Let $k,r,s,t$ be positive integers such that $r$ is sufficiently large, $s<t$ and $s\geq t(1-\frac{1}{6r})$. Let $G$ be a $K_{k,k}$-free bipartite graph with parts $A$ and $B$ such that $d_G(v)=r$ for every $v\in B$, $d_G(u)\leq 2^t$ for every $u\in A$ and $e(G)\geq 4(k+1)^22^s|A|$.
    
    Then $G$ contains a $64$-almost-regular subgraph with average degree at least $\frac{r}{160(k+1)\log r}$.
\end{lemma}

\begin{proof}
    We may assume that $r\geq k\log r$, otherwise the conclusion of the lemma is trivial. By Lemma~\ref{lem:clean large codegrees}, $G$ has a spanning subgraph $G'$ such that $e(G')\geq e(G)/(k+1)$ and $d_{G'}(u,u')\leq k 2^{(1-1/k)t}$ for any two distinct $u,u'\in A$. Since $s\leq t-1$ and $s\geq t(1-\frac{1}{6r})$, we have $t\geq 6r$. Hence, $2^{\frac{t}{2k}}\geq 2^{3r/k}\geq 2^{\log r}=r\geq k$. Thus, $k 2^{(1-1/k)t}\leq 2^{(1-\frac{1}{2k})t}$, so
    $d_{G'}(u,u')\leq 2^{(1-\frac{1}{2k})t}$ for any two distinct $u,u'\in A$. Let $\tilde{B}$ be the subset of $B$ consisting of those vertices $v$ which satisfy $d_{G'}(v)\geq r/(2k+2)$. Now the number of edges in $G'$ between $A$ and $B\setminus \tilde{B}$ is at most $|B|r/(2k+2)=e(G)/(2k+2)\leq e(G')/2$, so there are at least $e(G')/2$ edges between $A$ and $\tilde{B}$. Let $r'=\lceil r/(2k+2)\rceil$ and let $G''$ be a spanning subgraph of $G'[A\cup \tilde{B}]$ obtained by keeping precisely $r'$ edges from each $v\in \tilde{B}$. Clearly, $e(G'')\geq e(G'[A\cup \tilde{B}])/(2k+2)\geq e(G')/(4k+4)\geq e(G)/(4(k+1)^2)\geq 2^s|A|$.
    Note that
    \begin{align*}
        2r's-(2r'-1)t-r'
        &=t-r'(2t-2s+1)\geq t-\frac{r}{k}(2t-2s+1)\geq t-\frac{3r}{k}(t-s) \\
        &\geq t-\frac{3r}{k}\cdot \frac{t}{6r}=(1-\frac{1}{2k})t,
    \end{align*}
    so $d_{G''}(u,u')\leq 2^{2r's-(2r'-1)t-r'}$ holds for any distinct $u,u'\in A$.
    Thus, we can apply Lemma~\ref{lem:small codeg implies k-regular} with $G''$ in place of $G$ and $r'$ in place of $r$ to get a $64$-almost-regular subgraph with average degree at least $\frac{r'}{80\log r'}\geq \frac{r}{160(k+1)\log r}$.
\end{proof}

We are now in a position to prove our main result. In the proof and later in the paper, we omit floor and ceiling signs whenever they are not crucial.

\begin{theorem} \label{thm:general main result}
    Let $k$, $r$ and $\Delta$ be positive integers such that $r$ is sufficiently large. Let $G$ be a $K_{k,k}$-free graph with maximum degree at most $\Delta$ and average degree at least $80r^2\log \log \Delta$. Then $G$ has a $64$-almost-regular subgraph with average degree at least $\frac{r}{160(k+1)\log r}$.
\end{theorem}

\begin{proof}
    Since $r$ is sufficiently large, the degree conditions imply that $\Delta$ is also sufficiently large. Moreover, we may assume that $r\geq k$, otherwise the conclusion of the lemma is trivial.
    Note that $G$ has a bipartite subgraph $G'$ with average degree at least $40r^2\log \log \Delta$ and $G'$ has a non-empty subgraph $G''$ with minimum degree at least $20r^2\log \log \Delta$. Let $A$ and $B$ be the parts of $G''$ such that $|A|\leq |B|$. Let $H$ be a spanning subgraph of $G''$ such that $d_H(v)=20r^2\log \log \Delta$ for every $v\in B$.
    
    Let $t_0=r (\log r)(\log \log \Delta)^{1/2}$ and let $\ell$ be the smallest non-negative integer such that $t_0/(1-\frac{1}{10r})^{\ell}\geq \log \Delta$. Note that $(1-\frac{1}{10r})^{10r\log \log \Delta}\leq \exp(-\log \log \Delta)\leq 1/\log \Delta$, so $\ell\leq 10r\log \log \Delta$. For $1\leq i\leq \ell$, let $t_i=t_0/(1-\frac{1}{10r})^{i}$. Clearly, $t_\ell\geq \log \Delta$.
	
	Let $A_0=\{u\in A: d_H(u)\leq 2^{t_0}\}$ and for $1\leq i\leq \ell$, let $A_i=\{u\in A: 2^{t_{i-1}}<d_H(u)\leq 2^{t_i}\}$. Clearly, these sets partition $A$. Hence, for every $v\in B$, either $v$ has at least $d_H(v)/2= 10r^2\log \log \Delta$ neighbours (in the graph $H$) in $A_0$ or $\ell>0$ and there is some $1\leq i\leq \ell$ such that $v$ has at least $d_H(v)/(2\ell)\geq r$ neighbours in $A_i$.
	
	Therefore, at least one of the following two cases must occur.
	
	\noindent \textbf{Case 1.} There are at least $|B|/2$ vertices $v\in B$ which have at least $10r^2\log \log \Delta$ neighbours in $A_0$.
	
	\noindent \textbf{Case 2.} There exist some $1\leq i\leq \ell$ and at least $|B|/(2\ell)$ vertices $v\in B$ which have at least $r$ neighbours in $A_i$.
	
	\sloppy In Case 1, let $B'\subset B$ be a set of size at least $|B|/2$ such that every $v\in B'$ has at least $10r^2\log \log \Delta$ neighbours in $A_0$. For technical reasons, let us take a random subset $A_0'\subset A_0$ of size $|A_0|/3$. With positive probability, there is a set $B''\subset B'$ of at least $2|B'|/3$ vertices which all have at least $r^2\log \log \Delta$ neighbours in $A_0'$. Let $H'$ be a spanning subgraph of $H[A_0'\cup B'']$ obtained by keeping precisely $r^2\log \log \Delta$ edges from each vertex in $B''$. Now note that $|A_0'|=|A_0|/3\leq |A|/3\leq |B|/3\leq 2|B'|/3\leq |B''|$. Moreover, $d_{H'}(u)\leq d_H(u)\leq 2^{t_0}$ for every $u\in A_0'$. This implies that $H'$ is $(L,d)$-almost biregular for $L=2^{t_0}=2^{r(\log r)(\log \log \Delta)^{1/2}}$ and $d=r^2\log \log \Delta$. Hence, by Lemma \ref{lem:almost bireg to almost reg}, $H'$ has a $64$-almost-regular subgraph with average degree at least $\frac{r^2\log \log \Delta}{16r(\log r)(\log \log \Delta)^{1/2}}\geq \frac{r}{160(k+1)\log r}$.
	
	In Case 2, let us choose some $1\leq i\leq \ell$ and a set $B'\subset B$ of size at least $|B|/(2\ell)$ such that for every $v\in B'$, $v$ has at least $r$ neighbours in $A_i$. Let $H'$ be a subgraph of $H[A_i\cup B']$ obtained by keeping precisely $r$ edges incident to each $v\in B'$. We will apply Lemma \ref{lem:key lemma cor} for this graph. Let $t=t_i$ and let $s=t(1-\frac{1}{6r})$. Note that for any $u\in A_i$, $d_{H'}(u)\leq d_{H}(u)\leq 2^{t}$. Let $C=4(r+1)^2$. Using that $|B|\cdot 20r^2\log \log \Delta= e(H)\geq |A_i|2^{t_{i-1}}$, we get
	\begin{align*}
	    e(H')
	    &=|B'|r\geq |B|r/(2\ell)=\frac{e(H)}{40\ell r \log \log \Delta}\geq \frac{e(H)}{400r^2(\log \log \Delta)^2}\geq \frac{C|A_i|2^{t_{i-1}}}{400Cr^2(\log \log \Delta)^2} \\
	    &=C|A_i|2^{t_{i-1}-\log(400Cr^2(\log \log \Delta)^2)}.
	\end{align*}
	Note that $t_{i-1}-\log(400Cr^2(\log \log \Delta)^2)=t(1-\frac{1}{10r})-\log(400Cr^2(\log \log \Delta)^2)\geq t(1-\frac{1}{6r})=s$, where the inequality follows from $t/r\geq t_0/r=(\log r)(\log \log \Delta)^{1/2}$ and since $\Delta$ is sufficiently large.
	Hence, $e(H')\geq C|A_i|2^s\geq 4(k+1)^2|A_i|2^s$.
	
	Thus, we can apply Lemma \ref{lem:key lemma cor} to the graph $H'$ and get a $64$-almost-regular subgraph with average degree at least $\frac{r}{160(k+1)\log r}$.
\end{proof}

Theorem \ref{thm:logloglog error term} now follows easily.

\begin{proof}[Proof of Theorem \ref{thm:logloglog error term}]
    Let $r$ be sufficiently large in terms of $k$ and let $C=80r^2$. Let $G$ be a graph with maximum degree $\Delta$ and average degree at least $C\log \log \Delta$. If $G$ contains $K_{k,k}$ as a subgraph, then it has a $k$-regular subgraph. Else, by Theorem \ref{thm:general main result}, $G$ has a $64$-almost-regular subgraph $G'$ with average degree at least $\frac{r}{160(k+1)\log r}$. Since $r$ is sufficiently large in terms of $k$, Theorem \ref{thm:PRSz} implies that $G'$ has a $k$-regular subgraph.
\end{proof}

\section{Concluding remarks} \label{sec:concluding remarks}
Motivated by the study of the Tur\'an number of the cube, in 1970, Erd\H os and Simonovits~\cite{ES70} proved the following result.

\begin{theorem}[Erd\H os--Simonovits \cite{ES70}]
    For every $\alpha>0$ there exist some $K=K(\alpha)$ and $n_0=n_0(\alpha)$ such that any graph with $n\geq n_0$ vertices and at least $n^{1+\alpha}$ edges contains a $K$-almost-regular subgraph with $m$ vertices and at least $\frac{2}{5}m^{1+\alpha}$ edges for some $m\geq n^{\alpha(1-\alpha)/(1+\alpha)}$.
\end{theorem}

\noindent
This result has since become one of the most widely used tools for Tur\'an type problems. Its extreme usefulness comes from the fact that it allows us to replace a general host graph by an almost-regular one at negligible cost.
Usually extremal problems are much easier to deal with when the host graph is almost-regular.

In their paper, Erd\H os and Simonovits asked whether a similar ``regularization'' is possible for sparser graphs. More precisely, they asked whether there exist absolute constants $\eps,K>0$ such that any $n$-vertex graph with at least $n\log n$ edges contains a $K$-almost-regular subgraph with $m$ vertices and at least $\eps m\log m$ edges, where $m\rightarrow \infty$ as $n\rightarrow \infty$. Using a variant of the construction of Pyber, R\"odl and Szemer\'edi from Theorem~\ref{thm:PRSzlower}, Alon \cite{Alon08} gave a negative answer to this question as follows.

\begin{theorem}[Alon \cite{Alon08}] \label{thm:Alon sqrt}
    For every $K>0$ and $n>10^6$, there is an $n$-vertex graph with at least $n\log n$ edges in which any $K$-almost-regular $m$-vertex subgraph has at most $72m\sqrt{\log m}+18\log(64K)+324$ edges.
\end{theorem}

Thus we cannot necessarily pass to an almost-regular $m$-vertex subgraph with much more than $m\sqrt{\log m}$ edges. This naturally leads to the question what density we can actually guarantee in an almost-regular subgraph of a graph with $n\log n$ edges. Our Theorem \ref{thm:general main result} essentially answers this question. Indeed, it implies that we can always find an almost-regular $m$-vertex subgraph with nearly $m\sqrt{\log m}$ edges, showing that Theorem \ref{thm:Alon sqrt} is asymptotically tight.

\begin{theorem}
    For any positive integer $m_0$, there exist some $n_0=n_0(m_0)$ and $\eps=\eps(m_0)>0$ such that any graph with $n\geq n_0$ vertices and at least $n\log n$ edges has a $64$-almost-regular subgraph with $m\geq m_0$ vertices and at least $\eps m\sqrt{\log m}/(\log \log m)^{3/2}$ edges.
\end{theorem}

\begin{proof}
    Let $n_0$ be sufficiently large in terms of $m_0$ and let $G$ be a graph with $n\geq n_0$ vertices and at least $n\log n$ edges. Let $r=\frac{\sqrt{\log n}}{10\sqrt{\log \log n}}$ and let $\Delta$ be the maximum degree of $G$. Since $\Delta\leq n$, the average degree of $G$ is at least $\log n\geq 100r^2\log \log \Delta$. If $K_{m_0,m_0}$ is a subgraph of $G$, then we are done. Else, by Theorem \ref{thm:general main result}, $G$ has a $64$-almost-regular subgraph $G'$ with average degree at least $\frac{r}{160(m_0+1)\log r}$. This is at least $2\eps \sqrt{\log n}/(\log \log n)^{3/2}$ for some $\eps>0$ that depends only on $m_0$. Now if $m$ is the number of vertices in $G'$, then $m\geq m_0$ (since $n$ is sufficiently large) and $G'$ has at least $\eps m\sqrt{\log n}/(\log \log n)^{3/2}\geq \eps m\sqrt{\log m}/(\log \log m)^{3/2}$ edges, as desired.
\end{proof}

In certain extremal problems one wants to pass to almost-regular subgraphs whose average degree is large not just compared to the number of vertices in the subgraph but also compared to the number of vertices in the original graph (which we call $n$). Using a variant of Pyber's argument, it was shown in \cite{ARS+17} and \cite{BKPSTW20} that passing to such a subgraph is possible with the loss of only a $\log n$ factor in the average degree. Also, by considering a variant of the construction of Pyber, R\"odl and Szemer\'edi, it was shown in \cite{BKPSTW20} that this $\log n$ loss is necessary for graphs with average degree at least $n^{\eps}$. However, when the graph has at most $n\log n$ edges, then this $\log n$ loss means that the above result does not say anything non-trivial. Our methods apply in this sparse regime and give the following.

\begin{theorem} \label{thm:large average degree almost-regular subgraph}
    There is an absolute constant $\eps>0$ such that any $n$-vertex graph with average degree $d\geq 2\log \log n$ has a $64$-almost-regular subgraph with average degree at least $\eps \frac{(d/\log \log n)^{1/4}}{(\log (d/\log \log n))^{1/2}}$. Furthermore, for any positive integer $t$ there is some $\eps_t>0$ such that any $K_{t,t}$-free $n$-vertex graph with average degree $d\geq 2\log \log n$ has a $64$-almost-regular subgraph with average degree at least $\eps_t\frac{(d/\log \log n)^{1/2}}{\log (d/\log \log n)}$.
\end{theorem}

\begin{proof}
    We only prove the first assertion -- the second one has a very similar proof.
    
    Let $G$ be an $n$-vertex graph with average degree $d$. We may assume that $d/\log \log n$ is sufficiently large, otherwise the statement is trivial. Let $\eps$ be a sufficiently small positive absolute constant and let $k=\lceil \eps \frac{(d/\log \log n)^{1/4}}{(\log (d/\log \log n))^{1/2}}\rceil$. We want to show that $G$ has a $64$-almost-regular subgraph with average degree at least $k$. Let $r=(\frac{d}{80\log \log n})^{1/2}$. If $G$ contains $K_{k,k}$ as a subgraph, then we are done; else, by Theorem \ref{thm:general main result}, $G$ has a $64$-almost-regular subgraph with average degree at least $\frac{r}{160(k+1)\log r}$. But for sufficiently small $\eps$, we have $\frac{r}{160(k+1)\log r}\geq k$, so the proof is complete.
\end{proof}

Since the problem of finding regular subgraphs arises  naturally in many combinatorial settings, it seems very likely that our results will have further applications. 
We conclude the paper by discussing two such applications.

\begin{itemize}
\item
Our results can be used to answer a recent question posed by Alon et al. in \cite{ARS+17}. Motivated by applications to neural networks, in \cite{ARS+17} the authors defined the notion of an \emph{$\alpha$-multitasker} graph and asked (see the discussion after Theorem 1.3 in the full version of their paper \cite{ARSfull}) if there exists an $n$-vertex $\alpha$-multitasker with average degree $\Theta(\log n)$ for some $\alpha>0$, independent of $n$. Our Theorem~\ref{thm:large average degree almost-regular subgraph} (or Theorem~\ref{thm:logloglog error term}) can be used to show that there is no such multitasker with average degree $\omega(\log \log n)$. This is tight as Alon et al. showed that there are $\alpha$-multitaskers with average degree $\Theta(\log \log n)$ and constant $\alpha$.

\item
The second application concerns the notion of spectral degeneracy, introduced by Dvo\v{r}\'ak and Mohar \cite{DM12}. The spectral radius of a graph $G$ is the largest eigenvalue of the adjacency matrix of $G$ and is denoted by $\rho(G)$. We say that $G$ is \emph{spectrally $d$-degenerate} if for every subgraph $H$ of $G$, we have $\rho(H)\leq \sqrt{d\Delta(H)}$ and call the smallest such $d$ the \emph{spectral degeneracy of $G$}. It was observed in \cite{DM12} that every $d$-degenerate graph is spectrally $4d$-degenerate. Moreover, they showed the rough converse that any spectrally $d$-degenerate graph with maximum degree $D\geq 2d$ is $4d\log (D/d)$-degenerate. In addition, they asked whether there exists a function $f$ such that any spectrally $d$-degenerate graph is $f(d)$-degenerate. This was answered negatively by Alon \cite{Alon13}, who constructed, for every $M$, a spectrally $50$-degenerate graph which is not $M$-degenerate. Analyzing his construction more carefully, one can see that in fact there is a positive constant $c$ such that for every $n$ there is a spectrally $50$-degenerate $n$-vertex graph with degeneracy at least $c\log \log n$. Our results can be used to show that this is best possible, i.e. that for every $d$ there is a constant $C=C(d)$ such that any spectrally $d$-degenerate $n$-vertex graph 
has degeneracy at most $C\log \log n$. This follows from our Theorem \ref{thm:logloglog error term} and the simple fact that the spectral degeneracy of a $d$-regular graph is precisely $d$.
\end{itemize}

\bibliographystyle{abbrv}
\bibliography{bibliography}

\end{document}